\documentclass[reqno,12pt]{amsart}

\usepackage{amsmath}
\usepackage{amsfonts}
\usepackage{amssymb}
\usepackage{eufrak}
\usepackage{amscd}
\usepackage{amsthm}
\usepackage{amstext}
\usepackage[all]{xy}
\newcommand{\id}{\operatorname{id}}

\newcommand{\supp}{\operatorname{supp}}

   \theoremstyle{plain}%default
   \newtheorem{thm}{Theorem}[section]
   \newtheorem{prop}[thm]{Proposition}
   \newtheorem{lem}[thm]{Lemma}
   
   \theoremstyle{definition}

   \theoremstyle{remark}
   
   \newtheorem{remark}[thm]{Remark}

 \numberwithin{equation}{section}

\author{V. Manuilov}

\date{}

\address{Moscow State University,
Leninskie Gory 1, Moscow, 
119991, Russia}

\email{manuilov@mech.math.msu.su}

\title{On a family of representations of residually finite groups}

\begin{document}

\maketitle

\begin{abstract}
For a residually finite group $G$, its normal subgroups $G\supset G_1\supset G_2\cdots$ with $\cap_{n\in\mathbb N}G_n=\{e\}$ and for a growth function $\gamma$ we construct a unitary representation $\pi_\gamma$ of $G$. For the minimal growth, $\pi_\gamma$ is weakly equivalent to the regular representation, and for the maximal growth it is weakly equivalent to the direct sum of the quasiregular representations on the quotients $G/G_n$. In the case of intermediate growth we show two examples of different behaviour of $\pi_\gamma$.

\end{abstract}

\section*{Introduction}

Let $G$ be a residually finite group with normal subgroups $G\supset G_1\supset G_2\cdots$ such that $\cap_{n\in\mathbb N}G_n=\{e\}$. Let $X_n=G/G_n$ be the quotient finite group, with the natural action of $G$ denoted by $(g,x)\to gx$, $g\in G$, $x\in X_n$. Let $\nu_n=|X_n|$, where $|A|$ denotes the number of elements in the set $A$. A sequence $(\gamma_n)_{n\in\mathbb N}$ of integers such that $1\leq\gamma_n\leq \nu_n$ for any $n\in\mathbb N$ will be referred to as a {\it growth function}. 

It is well known that the regular representation $\lambda_G$ of $G$ is weakly contained in the direct sum of quasi-regular representations $\lambda_n$, which factorize through the quotients $G/G_n$, i.e. $\|\lambda_G(a)\|\leq\sup_{n\in\mathbb N}\|\lambda_n(a)\|$ for any $a\in \mathbb C[G]$. In view of abundance of exotic group norms \cite{Brown-Guentner,Wiersma}, we were interested to find intermediate norms between the norms $\|\lambda_G(\cdot)\|$ and $\sup_{n\in\mathbb N}\|\lambda_n(\cdot)\|$ (in the case of non-amenable $G$ --- otherwise these two norms coincide). We had not succeeded, but we have constructed a family $\pi_\gamma$, $\gamma\in Z$, of representations on Hilbert spaces $H_\gamma$, where $Z$ is a partially ordered set of growth functions, such that $\pi_\gamma$ is weakly equivalent to $\lambda$ for the minimal growth function ($\gamma_n=1$ for any $n\in\mathbb N$), and to $\oplus_{n\in\mathbb N}\lambda_n$ for the maximal growth function ($\gamma_n=\nu_n$ for any $n\in\mathbb N$). The most interesting case is that of intermediate growth. We prove an estimate on the norm of $\pi_\gamma$ when the growth function $\gamma$ is sufficiently slow. We also determine, in two cases, whether the trivial representation is weakly contained in $\pi_\gamma$. For a property $(\tau)$ group $G$ the answer is negative for all growth functions that are slower than the maximal one, but for a certain choice of finite index subgroups of a free group the answer is positive for some $\gamma$ of an intermediate growth. We were not able to check if the norms $\|\pi_\gamma(\cdot)\|$ are all different. We hope they are; if not, then there should be a boundary growth function such that slower growth functions give the regular norm, and faster ones give the norm $\sup_{n\in\mathbb N}\|\lambda_n(\cdot)\|$.

\section{Construction of a family of representations}

In this section we modify the Calkin's construction \cite{Calkin} 
to obtain representations of certain quotient $C^*$-algebras. 

Let $l^2(X_n)$, $n\in\mathbb N$, be the finitedimensional Hilbert spaces with
respect to the atomic measure $\mu$ given by $\mu(x)=1$ for any
$x\in X_n$. Set $\sigma=\oplus_{n\in\mathbb
N}\lambda_n$, and let $\mathcal A=C^*_{\sigma}(G)$ be the $C^*$-algebra generated by
all $\sigma(g)$, $g\in G$. Note that $\sigma$ is a representation of $G$ on the Hilbert space $\oplus_{n\in\mathbb N}l^2(X_n)$, and there is a canonical inclusion
$\mathcal A\subset\prod_{n\in\mathbb N}\mathbb B(l^2(X_n))\subset\mathbb B(\oplus_{n\in\mathbb N}l^2(X_n))$.

Fix some non-principal ultrafilter $\mathcal U$ on $\mathbb N$ and let $\tau_n$
be the trace on $\mathbb B(l^2(X_n))$ normalized by
$\tau_n(\id)=1$. Let $J\subset\prod_{n\in\mathbb N}\mathbb B(l^2(X_n))$ be the ideal consisting 
of all sequences $(m_n)_{n\in\mathbb N}$, $m_n\in\mathbb
B(l^2(X_n))$, such that $\lim_{\mathcal U}\tau_n(m_n^*m_n)=0$. Then
$I=J\cap \mathcal A$ is an ideal in $\mathcal A$.

Following Calkin, set 
$$
\widetilde{H}=\{\xi=(\xi_n)_{n\in\mathbb N}:\xi_n\in l^2(X_n),\
\sup_{n\in\mathbb N} \|\xi_n\|<\infty\}
$$ 
and define a sesquilinear
form on $\widetilde{H}$ by $\langle \xi,\eta\rangle=\lim_{\mathcal
U}\langle\xi_n,\eta_n\rangle$, where $\xi_n,\eta_n\in l^2(X_n)$,
$\xi=(\xi_n)_{n\in\mathbb N},\eta=(\eta_n)_{n\in\mathbb
N}\in\widetilde{H}$.

Set $\widetilde{H}_0=\{\xi\in\widetilde{H}:\lim_{\mathcal
U}\langle\xi,\xi\rangle=0\}\subset\widetilde{H}$.
When passing to the quotient linear space
$\widetilde{H}/\widetilde{H}_0$, the sesquilinear form becomes
positive-definite, hence defines an inner product. Completion of
$\widetilde{H}/\widetilde{H}_0$ with respect to the corresponding
norm is a (non-separable) Hilbert space $\widehat H$.

The sequence $\lambda_n$, $n\in\mathbb N$, defines a unitary representation $\widehat \lambda$ of the $C^*$-algebra
$A$ on $\widetilde H$ by $\widehat\lambda(g)(\xi_n)=(\lambda_n(g)\xi_n)$, and it is well-known due to Calkin \cite{Calkin} that
this representation annulates the ideal $I_0$ in $A$ consisting of
sequences $(m_n)_{n\in\mathbb N}$ that vanish at infinity, i.e.
$\lim_{n\to\infty}\|m_n\|=0$. 

This ideal is smaller than $I$. One of our aims is to find a subspace $H_1\subset \widehat H$ invariant under
$\widehat{\lambda}$ such that $\widehat{\lambda}|_{H_1}$ would annulate $I$.

Recall that $\nu_n=|X_n|$, and, by the definition, $\frac{\nu_{n+1}}{\nu_n}\geq 2$, so $\lim_{n\to\infty}\nu_n=\infty$. Let $Z$ denote the set of all non-decreasing integer-valued sequences $\gamma=(\gamma_n)_{n\in\mathbb N}$ such that $1\leq \gamma_n\leq \nu_n$ for any $n\in\mathbb N$. For $\gamma,\gamma'\in Z$, we write $\gamma\leq\gamma'$ if there is $\varepsilon>0$ such that $\varepsilon\gamma_n\leq\gamma'_n$ for any $n\in\mathbb N$, and $\gamma\sim\gamma'$ if $\gamma\leq\gamma'$ and $\gamma'\leq\gamma$. 
%With an abuse of notation we do not distinguish between elements of $Y$ and their representatives in $Z$.
The set $Z$ contains the minimal element $\iota$ such that $\iota_n=1$ for any $n\in\mathbb N$, and the maximal element $\nu$ ($\nu_n=|X_n|$ for any $n\in\mathbb N$). We write $\gamma\prec\gamma'$ if $\lim_{\mathcal U}\frac{\gamma_n}{\gamma'_n}=0$.

Fix $\gamma\in Z$.
Let $\widetilde H^{(k)}_\gamma\subset\widetilde H$ be the subset of all sequences $(\xi_n)$ such that
$|\supp \xi_n|\leq k\gamma_n$ for each $n$. This is not a linear subspace, but $\widetilde H_\gamma=\cup_{k=1}^\infty \widetilde H^{(k)}_\gamma$ is. Indeed, if $(\xi_n)\in\widetilde H^{(k)}_\gamma$, $(\xi'_n)\in\widetilde H^{(k')}_\gamma$ then
$(\xi_n+\xi'_n)\in\widetilde H^{(k+k')}_\gamma$.

Let $H_\gamma$ denote the closure of $(\widetilde{H}^{(k)}_\gamma+\widetilde H_0)/\widetilde H_0\subset\widehat H$. Note that if $\gamma,\gamma'\in Z$, $\gamma\leq\gamma'$, then $H_\gamma\subset H_{\gamma'}$, hence $H_\gamma=H_{\gamma'}$ if $\gamma\sim\gamma'$. If $\gamma=\nu$ then $H_\nu=\widehat H$.

\section{Some properties of the Hilbert spaces $H_\gamma$}

\begin{lem}
Let $\gamma\prec\gamma'$. Then $H_\gamma\neq H_{\gamma'}$.
\end{lem}
\begin{proof}
Assume the contrary, and let $E'_n\subset X_n$, $|E'_n|=\gamma'_n$. Take $\xi=(\xi_n)\in\widetilde{H}$, with $\xi_n=\frac{1}{\sqrt{\gamma'_n}}\chi_{E'_n}$, where $\chi_E$ denotes the characteristic function of the set $E$. Then, for any $\varepsilon\in(0,1)$ there exists $k\in\mathbb N$ and $\eta\in \widetilde H^{(k)}_{\gamma}$ such that $\|\xi-\eta\|<\varepsilon$. Let $E_n=\supp\eta_n$, $|E_n|\leq k\gamma_n$. Then
\begin{eqnarray*}
\frac{\gamma'_n-k\gamma_n}{\gamma'_n}&\leq& \frac{|E'_n\setminus E_n|}{\gamma'_n}=\sum_{x\in E'_n\setminus E_n}|\xi_n(x)|^2\\
&\leq&\sum_{x\in E'_n\setminus E_n}|\xi_n(x)|^2+\sum_{x\in E_n}|\xi_n(x)-\eta_n(x)|^2=\|\xi_n-\eta_n\|^2<\varepsilon^2, 
\end{eqnarray*}
hence $k\frac{\gamma_n}{\gamma'_n}>1-\varepsilon^2$. Then $0=k\lim_{\mathcal U}\frac{\gamma_n}{\gamma'_n}\geq 1-\varepsilon^2$, a contradiction with $\gamma\prec\gamma'$.

\end{proof}

\begin{remark}\label{Remark}
Note that the same argument shows that $H_\gamma$ is strictly greater than the closure of the union of all $H_{\gamma'}$ over all $\gamma'$ such that $\gamma'\prec\gamma$.

\end{remark}

\begin{lem}\label{orthogonal}
Let $\eta=(\eta_n)_{n\in\mathbb N}\in\widetilde H$. If $\eta\in \widetilde H_1$ and $\lim_{\mathcal U}\|\eta_n\|_\infty=0$ then $\eta\in\widetilde H_0$.

\end{lem}
\begin{proof}
It suffices to prove Lemma in the case when $\eta\in\widetilde H^{(k)}_1$ for some $k\in\mathbb N$. This means that $|\supp\eta_n|\leq k$ for any $n\in\mathbb N$. Then $\|\eta_n\|^2\leq k\|\eta_n\|_\infty^2$.

\end{proof}

%\begin{lem}\label{orthogonal2}
%Let $\eta=(\frac{1}{\sqrt{|X_n|}}1_{X_n})_{n\in\mathbb N}\in\widehat{H}$, and let $\gamma\prec d$, where $d_n=|X_n|$. Then %$\eta\perp H_\gamma$.

%\end{lem}
%\begin{proof}
%It suffices to prove that $\eta\perp\zeta$ for any $\zeta\in\widehat{K}^{(k)}_\gamma$. But %$\langle\eta,\zeta\rangle=\lim_{\mathcal U}\frac{1}{\sqrt{|X_n|}}\sum_{x\in X_n}\zeta_n(x)$, hence, by the Cauchy--Schwarz %inequality, $|\langle\eta,\zeta\rangle|\leq \lim_{\mathcal U}\sqrt{\frac{k\gamma_n}{|X_n|}}\|\zeta_n\|^2=\lim_{\mathcal %U}\sqrt{\frac{k\gamma_n}{|X_n|}}=0$. 

%\end{proof}

Remark \ref{Remark} shows that the family of Hilbert spaces is not lower semicontinuous. Now let us show that it is upper semicontinuous. 

\begin{prop}
For any $\gamma^0\in Z$ one has $H_{\gamma^0}=\cap_{\gamma\succ\gamma^0}H_\gamma$.

\end{prop}
\begin{proof}
Let $\xi\in \cap_{\gamma\succ\gamma^0}H_\gamma$ be represented by a sequence $(\xi_n)_{n\in\mathbb N}$. For any $\varepsilon>0$ and for any $n\in\mathbb N$, consider all $\eta_n\in l^2(X_n)$ such that $\|\xi_n-\eta_n\|<\varepsilon$. Among all these $\eta_n$ one can find $\eta'_n$ with minimal possible $|\supp\eta_n|$. Set $\alpha^\varepsilon_n=|\supp\eta'_n|$.

Since $\xi\in \cap_{\gamma\succ\gamma^0}H_\gamma$, for any $\varepsilon>0$ and for any $\gamma\succ\gamma^0$ there exists $k\in\mathbb N$ and $\zeta\in \widetilde H^{(k)}_\gamma$ such that $\|\xi-\zeta\|<\varepsilon$. In particular, $|\supp\zeta_n|\leq k\gamma_n$ for all $n\in\mathbb N$. 

It follows from $\|\xi-\zeta\|<\varepsilon$ that there exists $\mathbb A\in\mathcal U$ sich that $\|\xi_n-\zeta_n\|<\varepsilon$ for any $n\in\mathbb A$. By assumption, $\alpha^\varepsilon_n\leq|\supp\zeta_n|\leq k\gamma_n$ for any $n\in\mathbb A$. Thus, for any $\varepsilon>0$ and any $\gamma\succ\gamma^0$ there exists $k\in\mathbb N$ and $\mathbb A\in\mathcal U$ such that $\alpha^\varepsilon_n\leq k\gamma_n$ for any $n\in\mathbb A$.  

Suppose that $\lim_{\mathcal U}\frac{\alpha^\varepsilon_n}{\gamma^0_n}=\infty$ for some $\varepsilon>0$.
Then set $\gamma'_n=\sqrt{\alpha^\varepsilon_n\gamma^0_n}$. By assumption, $\lim_{\mathcal U}\frac{\gamma'_n}{\gamma^0_n}=\lim_{\mathcal U}\sqrt{\frac{\alpha^\varepsilon_n}{\gamma^0_n}}=\infty$, hence $\gamma'\succ\gamma^0$. As we have already shown, there exists $k\in\mathbb N$ and $\mathbb A\in\mathcal U$ such that $\alpha^\varepsilon_n\leq k\gamma'_n=k\sqrt{\alpha^\varepsilon_n\gamma^0_n}$ for any $n\in\mathbb A$, hence $\alpha^\varepsilon_n\leq k^2\gamma^0_n$. This contradicts our assumption. 

Thus, for any $\varepsilon>0$ we have $\lim_{\mathcal U}\frac{\alpha^\varepsilon_n}{\gamma^0_n}<\infty$, i.e. for any $\varepsilon>0$ there exists $C\in\mathbb R$ and $\mathbb A\in\mathcal U$ such that $\alpha^\varepsilon_n\leq C\gamma^0_n$ for any $n\in\mathbb A$. Set $\eta_n=\left\lbrace\begin{array}{cl}\eta'_n&\mbox{if\ }n\in\mathbb A;\\0&\mbox{if\ }n\notin\mathbb A.\end{array}\right.$ 
Then $\eta=(\eta_n)_{n\in\mathbb N}\in H_{\gamma^0}$, and $\|\xi-\eta\|=\lim_{\mathcal U}\|\xi_n-\eta_n\|<\varepsilon$. As $\varepsilon$ is arbitrary, this implies that $\xi\in H_{\gamma^0}$ and we are done. 

\end{proof}

\begin{lem}\label{invariant}
The subspaces $H_\gamma$ are invariant under $\widehat{\lambda}$.
\end{lem}
\begin{proof}
Obvious: translation by $g\in G$, of functions in $l^2(X_n)$ does not change the size of their supports.
\end{proof}

Thus, we can restrict the representation $\widehat\lambda$ to $H_\gamma$ for any $\gamma\in Z$. Denote this restriction by $\pi_\gamma=\widehat\lambda|_{H_\gamma}$. 

Note that if $\gamma\leq\gamma'$ then $H_\gamma\subset H_{\gamma'}$, hence the representation $\pi_{\gamma'}$ contains the representation $\pi_\gamma$.

\section{Case of maximal growth}

\begin{thm}
The representations $\pi_\nu$ and $\oplus_{n\in\mathbb N}\lambda_n$ are weakly equivalent. 

\end{thm}
\begin{proof}
When $\gamma=\nu$ is maximal, there is no restrictions on the size of supports of vectors in $\widetilde H$, so $\pi_\nu=\widehat\lambda$.
The classical result of J. W. Calkin \cite{Calkin} states that the kernel of the representation $\widehat\lambda$ of $\mathcal A$ on $\widehat H$ is $\mathcal A\cap\mathbb K(\oplus_{n\in\mathbb N}l^2(X_n))$ (recall that $\mathcal A$ is generated by $\oplus_{n\in\mathbb N}\lambda_n(g)$, $g\in G$). As $\lambda_n$ is a subrepresentation in any $\lambda_m$ with $m>n$, each $\lambda_n$ repeats infinitely in $\sigma=\oplus_{n\in\mathbb N}\lambda_n$, hence $\|\widehat\lambda(\cdot)\|=\sup_{n\in\mathbb N}\|\lambda_n(\cdot)\|$.

\end{proof}

\section{Case of very slow growth}

\begin{thm}
There exists a growth function $\gamma$ with $\lim_{n\to\infty}\gamma_n=\infty$ such that $\pi_{\gamma'}$ is weakly equivalent to the regular representation $\lambda$ of $G$ for any $\gamma'\leq\gamma$. In particular, $\pi_\iota$ is weakly equivalent to $\lambda$.

\end{thm}
\begin{proof}

It suffices to prove the theorem in the case when $G$ is finitely generated.
Let $q_n:G\to G/G_n=X_n$ denote the quotient maps, and let $l$ (resp. $l_n$) be the word length function on $G$ (resp. on $X_n$) with respect to a fixed set $S\subset G$ (resp. $q_n(S)\subset X_n$) of generators, and let $d(g,h)=l(g^{-1}h$ (resp. $d_n(x,y)=l_n(x^{-1}y)$) determine a left invariant metric on $G$ (resp. on $X_n$).   
Let $B_r\subset G$ denote the ball of radius $r$.
Let $a\in\mathbb C[G]$, $\supp a\subset B_r$. For any $\gamma\in Z$ and for any $\varepsilon>0$ there exists $\xi\in\widetilde H^{(k)}_\gamma$ such that $\|\xi_n\|=1$ for any $n\in\mathbb N$ and 
\begin{equation}\label{norme}
\|\pi_\gamma(a)\|^2<\|\pi_\gamma(a)\xi\|^2+\varepsilon.
\end{equation}
%hence there exists $\mathbb A\in\mathcal U$ such that
%$$
%\|\pi_\gamma(x)\|^2<\|\lambda_n(x)\xi_n\|^2+\varepsilon
%$$
%for any $n\in \mathbb A$.

Let $\supp\xi_n=A\subset X_n$. By assumption, $|A|\leq k\gamma_n$ for any $n\in\mathbb N$. Our aim is to show that although $A$ may be scattered on $X_n$, one can replace $\xi_n$ by another functions $\zeta_n$ such that $\|\pi_\gamma(a)\|^2<\|\pi_\gamma(a)\zeta\|^2+\varepsilon$, $\|\zeta_n\|=1$, $n\in\mathbb N$, and such that $\supp\zeta_n$ lies in a ball of a controlled radius. 

Decompose the set $A$ into a disjoint union of its subsets $A=A_1\sqcup\cdots\sqcup A_k$ such that
\begin{itemize}
\item[(A1)]
For any $x\in A_i$, $i=1,\ldots, k$, there exists $y\in A_i$ such that $\rho(x,y)<3r$;
\item[(A2)]
If $x\in A_i$, $y\in A_j$, $i\neq j$, then $\rho(x,y)\geq 3r$.
\end{itemize}

Set $\xi_n|_{A_i}=\eta_i$. Then $\|\eta_1\|^2+\cdots+\|\eta_k\|^2=\|\xi_n\|^2=1$ and, as $\supp x^*x\subset B_{2r}$, so by (A2), one has $\langle\lambda_n(a^*a)\eta_i,\eta_j\rangle=0$ when $i\neq j$, hence
$$
\|\lambda_n(a)\xi_n\|^2=\sum_{i=1}^k\langle\lambda_n(a^*a)\eta_i,\eta_i\rangle.
$$
Suppose that for any $i=1,\ldots,k$ one has 
\begin{equation}\label{assum}
\|\lambda_n(a^*a)\eta_i,\eta_i\|\leq (\|\lambda_n(a^*a)\|-\varepsilon)\|\eta_i\|^2.
\end{equation}
Then, summing (\ref{assum}) up, we get 
$$
\|\lambda_n(a)\xi_n\|^2=\sum_{i=1}^k\|\lambda_n(a^*a)\eta_i,\eta_i\|\leq(\|\lambda_n(a^*a)\|-\varepsilon)\sum_{i=1}^k\|\eta_i\|^2=\|\lambda_n(a^*a)\|-\varepsilon.
$$
Passing to the limit over $\mathcal U$, we get
$$
\|\pi_\gamma(a)\xi\|^2\leq\|\pi_\gamma(a)\|^2-\varepsilon,
$$
which contradicts (\ref{norme}). Thus, for any $n\in\mathbb N$, there exists $i$ such that 
$$
\|\lambda_n(a^*a)\eta_i,\eta_i\|\leq(\|\lambda_n(a^*a)\|-\varepsilon)\|\eta_i\|^2.
$$ 

Note that $|\supp\eta_i|=|A_i|\leq|A|\leq k\gamma_n$, hence, by (A1), there exists a ball $\widetilde B_{r'}(z_n)$ of radius $r'=k\gamma_n\cdot 3r$ centered at some $z_n\in X_n$, such that $\supp\eta_i\subset \widetilde B_{r'}(z_n)$ (we use tilde to distinguish balls in $X_n$ from balls in $G$).  

Set $\zeta_n=\frac{\eta_i}{\|\eta_i\|}$. Then $\supp\zeta_n\subset \widetilde B_{r'}(z_n)$, $|\supp\zeta_n|\leq k\gamma_n$ for any $n\in\mathbb N$, and there exists $\mathbb A\in\mathcal U$ such that $\|\lambda_n(a)\zeta_n\|>\|\lambda_n(a)\|-\varepsilon$ for any $n\in\mathbb A$.

Note that $y\in \widetilde B_{r'}(z_n)$ if $l_n(z_n^{-1}y)\leq r'$. 

Let $\alpha_n=\min_{g\in G_n,g\neq e}l(g)$. It follows from $\cap_{n=1}^\infty G_n=\{e\}$ that $\lim_{n\to\infty}\alpha_n=\infty$.
The following statement is folklore.
\begin{lem}
Let $g\in G$, $z=q_n(g)$, and let $B_{R}(g)\subset G$ be the ball of radius $R$ centered at $g$. Then $q_n$ maps $B_R(g)$ isometrically onto $\widetilde B_R(z)$ for any $R<\alpha_n/4$.

\end{lem}
\begin{proof}
Let $h,k\in B_R(g)$. Then $d(h,k)=l(h^{-1}k)\leq 2R$. Note that for any $x\in G$, $l(x)\geq l(q_n(x))$, and $l_n(q_n(x))=\min_{y\in G_n,y\neq e}l(xy)$. Suppose that there exists $y\in G_n$ such that $l(xy)<l(x)$. Then, by the triangle inequality, $l(xy)\geq l(y)-l(x)$, hence $l(y)-l(x)<l(x)$, or, equivalently, $l(x)>l(y)/2$. Taking $x=h^{-1}k$, we get $2R\geq d(h,k)=l(x)>l(y)/2\geq \alpha_n/2$, hence $R\geq \alpha_n/4$ --- a contradiction. Thus, $d_n(q_n(h),q_n(k))=l_n(q_n(x))=l(x)=d(h,k)$.   

\end{proof}

Let $g_n\in G$ satisfy $q_n(g_n)=z_n$, $R<\alpha_n/4$, and let $u_n:l^2(\widetilde B_R(z_n))\to l^2(G)$ be an isometry defined by 
$$
u_n(\varphi_n)(h)=\left\lbrace\begin{array}{cl}\varphi_n(q_n(h))&\mbox{\ if\ }h\in \widetilde B_R(g_n);\\0&\mbox{\ if\ }h\notin \widetilde B_R(g_n).\end{array}\right.
$$
If the supports of $\varphi_n$ and of $\lambda_n(g)\varphi_n$ lie in $\widetilde B_R(z_n)$ then $u_n\lambda_n(g)\varphi_n=\lambda(g)u_n\varphi_n$. Both $\zeta_n$ and $\lambda_n(x)\zeta_n$ have supports in the ball $\widetilde B_{r'+r}(z_n)$, hence, for $r'+r<\alpha_n/4$ one has $u_n\lambda_n(x)\zeta_n=\lambda(x)u_n\zeta_n$, so $\|\lambda_n(x)\zeta_n\|=\|\lambda(x)\zeta'_n\|$, where $\zeta'_n=u_n\zeta_n\in l^2(G)$.  

Thus, 
\begin{equation}\label{es1}
\|\pi_\gamma(a)\|^2<\|\lambda_n(a)\zeta_n\|^2+\varepsilon=\|\lambda(a)\zeta'_n\|^2+\varepsilon
\end{equation} 
for any $n\in\mathbb A$ when $r'+r<\alpha_n/4$. 

For any $\varepsilon>0$, there exists $\mathbb A'\in\mathcal U$ such that 
\begin{equation}\label{es2}
\|\lambda(a)\zeta'_n\|^2<\|\lambda(a)\|^2+\varepsilon
\end{equation}
for any $n\in\mathbb A'$. 

Choose $\gamma$ such that
$\lim_{\mathcal U}\frac{\gamma_n}{\alpha_n}=0$. Then, as $r'=3rk\gamma_n$, so there exists $\mathbb A''\in\mathcal U$ such that $r'+r< \alpha_n/4$ holds for any $n\in\mathbb A''$. As the set $\mathbb A\cap\mathbb A'\cap\mathbb A''$ is not empty, it follows from (\ref{es1}) and (\ref{es2}) that 
$$
\|\pi_\gamma(a)\|^2<\|\lambda(a)\|^2+2\varepsilon
$$
holds for any $a\in\mathbb C[G]$ with $\supp a\in B_r$. As $r$ is arbitrary, we conclude that $\|\pi_\gamma(a)\|\leq\|\lambda(a)\|$ for any $a\in\mathbb C[G]$.

\end{proof}

\section{Case of small growth}

Let $U_{n-1}:l^2(X_{n-1})\to l^2(X_n)$ be the isometry defined by 
$$
U_{n-1}(\xi_{n-1})(x)=\frac{1}{\sqrt{\nu_n/\nu_{n-1}}}\xi_{n-1}(q_n(x)), 
$$
where $q_n:G/G_n\to G/G_{n-1}$ is the quotient homomorphism and $\nu_n=|X_n|$. This allows to consider $H_{n-1}$ as a subspace of $H_n$. Note that $\lambda_{n-1}$ is a subrepresentation of $\lambda_n$.
Set $\rho_n=\lambda_n\ominus\lambda_{n-1}$. This is a representation of $G$ on $l^2(X_n)\ominus l^2(X_{n-1})$.

\begin{thm}
Let $\gamma\prec\gamma'$, where $\gamma'_n=\frac{\nu_n}{\nu_{n-1}}$. Then $\|\pi_\gamma(a)\|\leq \lim\sup_{n\to\infty}\|\rho_n(a)\|$ for any $a\in\mathbb C[G]$.

\end{thm}
\begin{proof}
We are going to construct a Hilbert space $L$ such that $H_\gamma\subset L\subset H_\nu=\widehat H$ and a representation $\rho$ of $G$ on $L$ such that $\rho=\widehat\lambda|_L$ and $\|\rho(a)\|=\lim\sup_{n\to\infty}\|\rho_n(a)\|$ for any $a\in\mathbb C[G]$. This would obviously imply that $\rho$ contains $\pi_\gamma$, hence the claim of the theorem. 

In order to construct $L$ let us consider the set $\widetilde L$ of all the sequences $(\xi_n)_{n\in\mathbb N}$ such that $\xi_n\in l^2(X_n)\ominus l^2(X_{n-1})$ and the norms $\|\xi_n\|$ are uniformly bounded, with the degenerate inner product as before, which becomes positive definite after taking quotient modulo sequences with $\lim_{\mathcal U}\|\xi_n\|^2=0$. Then $\widetilde L\subset\widetilde H$, and we define $L$ as the closure of $\widetilde L+\widetilde H_0/\widetilde H_0$ in $\widehat H$. Then $\rho(g)(\xi_n)=(\rho_n(g)\xi_n)$, $g\in G$, defines the representation $\rho$ on $L$.   

Let $\sigma_k$, $k\in\mathbb N$, be the sequence of irreducible representations of $G$ that appear as direct summands in $\lambda_1,\lambda_2,\ldots$. Let $\sigma_k$ be a subrepresentation of $\lambda_{n-1}$. As the multiplicity of $\sigma_k$ in $\lambda_{n-1}$ and in $\lambda_n$ is the same and equals its dimension, it is not contained in $\rho_n$. Thus, each $\sigma_k$ appears only in one of $\rho_1,\rho_2,\ldots$. Let $\mathcal B$ be the $C^*$-algebra generated by all $\oplus_{n\in\mathbb N}\rho_n(g)$, $g\in G$, in $H'=\oplus_{n\in\mathbb N}l^2(X_n)\ominus l^2(X_{n-1})$. Then $\|\rho(a)\|$ equals the norm of $a$ in $\mathcal B/\mathcal B\cap\mathbb K(H')$ by Theorem 5.3 of \cite{Calkin}. Since each $\sigma_k$ appears only in one of $\rho_1,\rho_2,\ldots$, the latter norm equals $\lim\sup_{k\to\infty}\|\sigma_k(a)\|=\lim\sup_{n\to\infty}\|\rho_n(a)\|$.

To finish the argument, it remains to show that $H_\gamma\subset L$. Let $P_n:l^2(X_n)\to l^2(X_n)\ominus l^2(X_{n-1})$ be the orthogonal projection, $Q_n=1-P_n$. We claim that $\lim_{\mathcal U}\|\xi_n-P_n\xi_n\|=0$ for any $\xi\in H_\gamma$. Let $\xi\in\widetilde H^{(k)}_\gamma$ for some $k\in\mathbb N$. 

Let $x\in X_{n-1}=G/G_{n-1}$, $\delta_x$ its characteristic delta-function, $\eta_x=U_{n-1}\delta_x\in l^2(X_n)$. Then $Q_n\xi_n=\sum_{x\in X_{n-1}}\langle\eta_x,\xi_n\rangle\eta_x$. Note that $\eta_x(y)=\left\lbrace\begin{array}{cl}\frac{1}{\sqrt{\nu_n/\nu_{n-1}}}&\mbox{\ if\ }q_n(y)=x;\\0&\mbox{\ if\ }q_n(y)\neq x.\end{array}\right.$

Note also that $\supp\eta_x\cap\supp\eta_y=\emptyset$ when $x\neq y$. Let $A_x=\supp\xi_n\cap\supp\eta_x\subset X_n$. Then $\supp\xi_n=\sqcup_{x\in X_{n-1}}A_x$, hence $\sum_{x\in X_{n-1}}|A_x|\leq k\gamma_n$.
\begin{eqnarray*}
\|Q_n\xi_n\|^2&=&\sum_{x\in X_{n-1}}|\langle\eta_x,\xi_n\rangle|^2\leq \sum_{x\in X_{n-1}}\sum_{y\in A_x}|\eta_x(y)|^2\sum_{y\in A_x}|\xi_n(y)|^2\\
&\leq& \sum_{x\in X_{n-1}}|A_x|\frac{1}{\nu_n/\nu_{n-1}}\|\xi_n\|^2\leq k\gamma_n\frac{\nu_{n-1}}{\nu_n}\|\xi_n\|^2.
\end{eqnarray*}

So, $\lim_{\mathcal U}\|Q_n\xi_n\|=0$, hence $\xi\in \widetilde L$.

\end{proof}

\section{Case of intermediate growth}

Here we consider the case of intermediate growth and show two different kinds of behaviour of representations $\pi_\gamma$ --- for property (T) groups and for free groups.

\subsection{Case of property $(\tau)$ groups}

Recall that the property $(\tau)$ is a generalization of the property (T) of Kazhdan and means that the trivial representation is isolated among finitedimensional representations. For more details about property $(\tau)$ we refer to \cite{Lubotzky}.

\begin{thm}
Let $G$ be a finitely generated property $(\tau)$ group, and let $\gamma\prec \nu$. Then the trivial representation is not weakly contained in $\pi_\gamma$.  

\end{thm}
\begin{proof}
Let $S\subset G$ be a finite symmetric generating set, and let $x=\frac{1}{|S|}\sum_{g\in S}g\in\mathbb C[G]$.
Suppose the contrary: for any $\varepsilon>0$ there exists $\xi^{(k)}\in\widetilde H_\gamma$ such that $\|\xi^{(k)}\|=1$ and $\|\pi_\gamma(x)\xi^{(k)}-\xi^{(k)}\|<\varepsilon$. Without loss of generality we may assume that $\xi^{(k)}\in\widetilde H^{(k)}_\gamma$, and that $\|\xi_n\|=1$ for each $n\in\mathbb N$, where $\xi=(\xi_n)_{n\in\mathbb N}$. A fixed $\varepsilon$ determines $k$ such that 
$$
\|\pi_\gamma(x)\xi^{(k)}-\xi^{(k)}\|=\lim_{\mathcal U}\|\lambda_n(x)\xi^{(k)}_n-\xi^{(k)}_n\|<\varepsilon.
$$
Then there exists $\mathbb A\in\mathcal U$ such that
\begin{equation}\label{eqq1}
\|\lambda_n(x)\xi^{(k)}_n-\xi^{(k)}_n\|<\varepsilon
\end{equation}
for any $n\in\mathbb A$.

Note that each $\lambda_n$ contains exactly one copy of the trivial representation, with the representation space spanned by the single vector $\xi^0_n=\frac{1}{\sqrt{|X_n|}}\chi_{X_n}$. By property $(\tau)$, there exists $\delta>0$ (which does not depend on $n$) such that $\|\lambda_n(x)\eta_n-\eta_n\|\geq \delta\|\eta_n\|$ for any $\eta_n\in l^2(X_n)$ orthogonal to $\xi^0_n$. Let $\xi^{(k)}_n=\alpha\xi^0_n+\eta_n$, where $\eta_n\perp \xi^0_n$. Then 
$$ 
\delta\|\eta_n\|\leq \|\lambda_n(x)\eta_n-\eta_n\|
=\|\lambda_n(x)\xi^{(k)}_n-\xi^{(k)}_n\|<\varepsilon,
$$
for any $n\in\mathbb A$, hence $\|\eta_n\|<\frac{\varepsilon}{\delta}$ when $n\in\mathbb A$. Then $|\alpha|=\sqrt{1-\|\eta_n\|^2}>\sqrt{1-\frac{\varepsilon^2}{\delta^2}}$, and 
\begin{equation}\label{estimate1}
\|\xi^{(k)}_n-\xi^0_n\|=\sqrt{(1-\alpha)^2+\|\eta_n\|^2}<\sqrt{2\frac{\varepsilon}{\delta}}
\end{equation}
when $n\in\mathbb A$ and $\varepsilon\leq\delta$.

But $|\supp\xi^{(k)}_n|\leq k\gamma_n$ for any $n\in\mathbb N$, hence $|\langle\xi^{(k)}_n,\xi^0_n\rangle|\leq\sqrt{\frac{k\gamma_n}{|X_n|}}$. So, by assumption, $\lim_{n\to\infty}\langle\xi^{(k)}_n,\xi^0_n\rangle=0$, which contradicts (\ref{estimate1}) when $\varepsilon<\delta/4$.

\end{proof}

\subsection{Case of free groups}

Let $G=\mathbb F_2$ be the free group on two generators. Here we show that $H_\gamma$ may weakly contain
the trivial representation for a certain sequence of finite index subgroups and for certain intermediate $\gamma$. We follow here \cite{Arzhantseva-Guentner}.

Let $G_0=\mathbb F_2$, $G_1=\mathbb F_2^{(2)}$, $G_{n+1}=G_n^{(2)}$, be the iterated subgroups generated by the squares of all elements of the previous group. There is a nice description of $X_n=\mathbb F_2/G_n$ in \cite{Arzhantseva-Guentner} as vertices of Cayley graphs of $X_n$. The Cayley graph of $\mathbb F_2/G_0$ is the wedge of two circles with a single vertex, and the Cayley graphs $Cay(X_n)$ of $\mathbb F_2/G_n=X_n$ are constructed from it inductively. Let $V_n$ and $E_n$ denote the set of vertices and edges of the Cayley graph of $X_n$, let $T_n$ be a maximal tree in $Cay(X_n)$, and let $e_1,\ldots,e_{r_n}$ be the edges not in $T_n$. Then $V_{n+1}=V_n\times(\mathbb Z/2)^{r_n}$, $E_{n+1}=E_n\times(\mathbb Z/2)^{r_n}$. Let $e\in E_n$, $\alpha\in(\mathbb Z/2)^{r_n}$. Let $e$ connect the vertices $v,w\in V_n$. If $e\in T_n$ then the edge $(e,\alpha)\in E_{n+1}$ connects $(v,\alpha)$ with $(w,\alpha)$. If $e=e_i$, $1\leq i\leq r_n$, then $(e,\alpha)$ connects $(v,\alpha)$ with $(w,\alpha+\bar e_i)$, where $\bar e_i=(0,\ldots,0,1,0,\ldots,0)\in(\mathbb Z/2)^{r_n}$ has 1 as its $i$'s component.

For a Cayley graph $Cay(X_n)$ and for a subset $A$ of vertices of $Cay(X_n)$ we denote by $|\partial A|$ the number of edges in $Cay(X_n)$ such that they connect a point from $A$ with a point from $X_n\setminus A$. 

\begin{lem}\label{lem-free}
There exists a sequence $A_n\subset X_n$ such that 
\begin{equation}\label{l1}
\lim_{n\to\infty}|A_n|=\infty; 
\end{equation}
\begin{equation}\label{l3}
\lim_{n\to\infty}\frac{|A_n|}{|X_n|}=0;
\end{equation}
\begin{equation}\label{l2}
\lim_{n\to\infty}\frac{|\partial A_n|}{|A_n|}=0. 
\end{equation}

\end{lem}
\begin{proof}
Let $1<k_n<r_n$. Set 
$$
A_{n+1}=\{(v,\alpha):v\in V_n,\alpha\in(\mathbb Z/2)^{r_n}, \alpha_i=0\mbox{\ for\ }k_n<i\leq r_n\}. 
$$
Then $|A_{n+1}|=|X_n|2^{k_n}$, and (\ref{l1}) holds for any $k_n\geq 1$. Note that (\ref{l3}) means that 
$$
\lim_{n\to\infty}\frac{|A_{n+1}|}{|X_{n+1}|}=\lim_{n\to\infty}2^{k_n-r_n}=0, 
$$
which is equivalent to $\lim_{n\to\infty}(r_n-k_n)=\infty$.  

Let us evaluate $|\partial A_{n+1}|$. If $e\in T_n$ then both ends of $(e,\alpha)$ are either in $A_{n+1}$ or in $V_{n+1}\setminus A_{n+1}$, so let $e=e_i$ be one of $e_1,\ldots,e_{r_n}$. If $(e,\alpha)\in\partial A_{n+1}$ then $i>k_n$, so $|\partial A_{n+1}|=2(r_n-k_n)2^{k_n}$.

Thus, $\frac{|\partial A_{n+1}|}{|A_{n+1}|}=\frac{2(r_n-k_n)}{|X_n|}$. As both $(r_n)_{n\in\mathbb N}$ and $(|X_n|)_{n\in\mathbb N}$ grow faster than an iterated exponential, so one can easily find a sequence $(k_n)_{n\in\mathbb N}$ such that
\begin{itemize}
\item
$\lim_{n\to\infty}(r_n-k_n)=\infty;$ 
\item
$\lim_{n\to\infty}\frac{(r_n-k_n)}{|X_n|}=0$ 
\end{itemize}
(the latter implies (\ref{l2})).    

\end{proof}

\begin{thm}
There exists $\gamma$ with $\lim_{n\to\infty}\gamma_n=\infty$ and $\lim_{n\to\infty}\frac{\gamma_n}{d_n}=0$ such that $\pi_\gamma$ weakly contains the trivial representation.

\end{thm}
\begin{proof}
Let $A_n\subset X_n$ be as in Lemma \ref{lem-free}.
Set $\gamma_n=|A_n|$, and $\xi_n=\frac{1}{\sqrt{\gamma_n}}\chi_{A_n}$. Then $\xi=(\xi_n)_{n\in\mathbb N}\in \widetilde H_\gamma$. If $g\in S\subset G$ is one of the generators then 
\begin{eqnarray*}
\|\lambda_n(g)\xi_n-\xi_n\|^2&\leq&\frac{1}{\gamma_n}(|g(A_n)\setminus A_n|+|A_n\setminus g(A_n)|)\\
&=&\frac{1}{\gamma_n}(|g(A_n)\setminus A_n|+|g^{-1}(A_n)\setminus A_n|)\leq\frac{2|\partial A_n|}{\gamma_n}, 
\end{eqnarray*}
hence $\xi$ is almost invariant for $\pi_\gamma$ by Lemma \ref{lem-free}.  

\end{proof}


\begin{thebibliography}{9}


%\bibitem{Blackadar-Kirchberg}
%B. Blackadar, E. Kirchberg. {\it Generalized inductive limits of
%finite-dimensional $C^*$-algebras.} Math. Ann. {\bf 307} (1997),
%343--380.

\bibitem{Arzhantseva-Guentner}
G. Arzhantseva, E. Guentner. {\it Coarse non-amenability and covers with small eigenvalues.} Math. Ann. {\bf 354} (2012), 863--870.


\bibitem{Brown-Guentner}
N. P. Brown, E. P. Guentner. {\it New $C^*$-completions of discrete groups and related spaces.} Bull. Lond. Math. Soc. {\bf 45} (2013), 1181--1193.

\bibitem{Calkin}
J. W. Calkin. {\it Two-sided ideals and congruences in the ring of
bounded operators in Hilbert space.} Ann. Math. {\bf 42} (1941),
839--873.

\bibitem{Lubotzky}
A. Lubotzky. Discrete groups, expanding graphs and invariant measures. Progress in Math. {\bf 125}, Birkh\"auser Verlag, Basel (1994).

%\bibitem{Paschke}
%W. Paschke. {\it The flow space of a directed $G$-graph.} Pacific
%J. Math. {\bf 159} (1993), 127--139.

\bibitem{Wass4}
S. Wassermann. {\it On tensor products of certain group
$C^*$-algebras.} J. Funct. Anal. {\bf 23} (1976), 239--254.

\bibitem{Wiersma}
M. Wiersma. {\it Constructions of exotic group $C^*$-algebras.} Illinois J. Math. {\bf 60} (2016), 655--667.


%\bibitem{Kirchberg}
%E. Kirchberg. {\it On semi-split extensions, tensor products and
%exactness of $C^*$-algebras}, Invent. Math. {\bf 112} (1993),
%449--489.



\end{thebibliography}
\end{document}